\documentclass[11pt]{amsart}
\usepackage[margin=3cm]{geometry}
\usepackage{amsfonts}
\usepackage{amsmath}
\usepackage{amsthm}
\usepackage{amssymb}
\usepackage{mathrsfs}
\usepackage{lipsum}
\usepackage{tikz-cd}
\usepackage{framed}

\usepackage{newclude}
\usepackage{nomencl}
\makenomenclature
\usepackage{varioref}
\usepackage{hyperref}
\usepackage{cleveref}
\usepackage{color}
\usepackage{comment}

\newenvironment{psmallmatrix}
  {\left(\begin{smallmatrix}}
        {\end{smallmatrix}\right)}

\numberwithin{equation}{section}

\newtheorem{thm}{Theorem}[section]

\newtheorem{defn}[thm]{Definition}
\newtheorem{assumption}[thm]{Assumption}
\newtheorem{comm}[thm]{RA:Comment}
\newtheorem{suggestion}[thm]{RA:Suggestion}
\newtheorem{question}[thm]{Question}

\newtheorem{ex}[thm]{Example}

\newcommand{\inlinemaketitle}{{\let\newpage\relax\maketitle}}
\newcommand{\be}{\begin{equation}}
\newcommand{\ee}{\end{equation}}
\def\eps{\varepsilon}

\def\NN{{\mathbb N}}
\def\dualSU2{\frac12\NN_0}
\def\RR{{\mathbb R}}
\def\ZZ{{\mathbb Z}}
\def\HH{\mathbb{H}}
\def\C{{\mathbb C}}

\def\Gh{{\widehat{G}}}

\def\H{{\mathcal H}}

\def\O{\mathcal{O}}
\def\g{\mathfrak{g}}

\def\Ad{\mathrm{Ad}}
\def\ad{\mathrm{ad}}

\def\SU2{{\rm SU_2}}

\def\SL2{{\rm SL_2}}
\def\su2{\mathfrak{{ su_2}}}
\def\sl2{\mathfrak{{ sl_2}}}
\def\G{{\widehat{G}}}

\DeclareMathOperator{\Id}{Id}

\DeclareMathOperator{\vol}{vol}

\begin{document}

\title{Contractions of group representations via geometric quantisation}
\author{Rauan Akylzhanov}
\address{Rauan Akylzhanov:
  \endgraf
  Department of Mathematics
  \endgraf
  Imperial College London
  \endgraf
  180 Queen's Gate, London SW7 2AZ
  \endgraf
  United Kingdom
  \endgraf
  {\it E-mail address} {\rm r.akylzhanov@imperial.ac.uk}
}

\thanks{}
\author{Alexis Arnaudon}
\address{Alexis Arnaudon:
  \endgraf
  Department of Mathematics
  \endgraf
  Imperial College London
  \endgraf
  180 Queen's Gate, London SW7 2AZ
  \endgraf
  United Kingdom
  \endgraf
  {\it E-mail address} {\rm a.arnaudon@imperial.ac.uk}
}

\thanks{
 }
\date{\today}


\begin{abstract}
We propose a general framework to contract unitary dual of Lie groups via holomorphic quantization of their co-adjoint orbits. 
The sufficient condition for the contractibility of a representation is expressed via cocycles on coadjoint orbits. This condition is checked explicitly for the contraction of $SU_2$ into $\HH$.
The main tool is the geometric quantization.  
We construct two types of contractions that can be implemented on every matrix Lie group with diagonal contraction matrix.
\end{abstract}

%\tableofcontents
\maketitle
\section{Introduction}

Physicists study certain degenerations of group laws which are called contractions.
In this paper, we study Lie group contractions in the context of geometric quantization.
The degeneration is understood in the sense that group law becomes "more abelian", i.e. more structure constants of Lie algebra becomes zero. The pioneering work of In\"on\"u and Wigner \cite{Inonu:1953ab} is a classical original source. Their motivation was simply to relate two different physics: Galilean and Lorentzian.
Contractions of representation of Lie groups and Lie algebras and their unitary duals have attracted the attention of many authors 
\cite{Mickelsson:1972aa,Ricci:1986aa,Ricci:1986ab,Cahen:2003aa,Dooley:1983aa, Higson:2008ab,Higson:2011ab} and continue to be a field of active investigations, see for example \cite{Fialowski:2008ab, Gromov:2004ab,Bernstein:2016ab, Bernstein:2017ab}.

The first general global definition of Lie group contraction has been given in \cite{Mickelsson:1972aa}where the authors study the contraction of the principal continuous series of the de Sitter groups $SO(n,1)$ to positive mass representations of both the Euclidean group $E(n)$ and the Poincar\'e group $E_0(n-1,1)$.
In \cite{Ricci:1986aa}, a contraction of $SU(2)$ to the Heisenberg group $\HH$ was studied in detail. The authors showed that the group action on the manifold depends smoothly on the contraction parameter $\eps>0$. They also established transference results linking Fourier analysis on these two groups.
Applying \cite{Ricci:1986aa}, Ricci and Rubin proved in \cite{Ricci:1986ab} an analogue of de Leeuw's theorem relating $L^p$-norms of Fourier multipliers on $SU(2)$ and $\HH$.
A short and elegant proof of the main result of \cite{Ricci:1986aa}  was given in \cite{Cahen:2003aa} employing Berezin quantization.
Furthermore, it was shown by \cite{A.H.Dooley:1985aa} that every semisimple Lie group can be contracted to its Cartan motion group. 

Generally, one cannot expect that there always exists a contraction $\Phi_{\eps}\colon G_0\to G_1$ for any Lie group $G_0$ and $G_1$. Indeed, let $\g_1$ and $\g_0$ be the Lie algebras of $G_1$ and $G_0$ and $U_{\eps}\colon \g_1\to \g_0$ be differential of $\Phi_{\eps}$. Then the existence of contraction $U_{\eps}$ is equivalent to the fact that $\g_0$ is the limit in the Zarissky topology on the affine variety of Lie algebras. It can be seen that not every two points in the affine variety of Lie algebras of dimension less than $6$ are path-connected in the Zarissky topology \cite[Theorem 1]{Kirillov:1987aa}.
In other words, the existence of Lie algebra contraction $U_{\eps}\colon \g_0\to \g_1$ is equivalent to saying that $\g_0$ is the limit of Lie algebras $U_{\eps}\g_1$ in the Zarissky topology. Hence, we shall assume that there exists a Lie algebra contraction $U_{\eps}$ of $\g_1$ into $\g_0$.

The aim of the paper is to show that the existence of Lie group contraction combined with the geometric quantization yields a systematic approach to link harmonic analysis on two groups.  
The novelty of the paper is that we do not impose any restriction on the group nor on the contraction.
In addition, we tried to answer the following question.
\begin{question} 
\label{quest}
Let $G_0$ and $G_1$ be two Lie groups and let $\Phi_{\eps}\colon G_1\to G_0$ be a Lie group contraction. Does there exist a contraction $\widehat{\Phi_{\eps}}$ of $\Gh_1$ into $\Gh_0$. In other words, does the following diagram commute
\begin{equation}
\begin{tikzcd}
  G_1 \arrow{r} {GQ} \arrow{d}{\Phi_{\eps}}
    & \widehat{G_1} \arrow{d} {\widehat{\Phi}_{\eps}} \\
  G_0 \arrow{r} {GQ}
&\widehat{G_0} \end{tikzcd}
\end{equation}
where $GQ$ stands for the geometric quantization.
We shall give meaning to ${\widehat{\Phi}_{\eps}}$ as an MN-contraction of $\widehat{G_1}$ into $\widehat{G_0}$.
\end{question}

A similar problem has been treated in \cite{Bernstein:2017aa} in the context of the algebraic Harish-Chandra modules.
%More generally, in the light of the Baum-Connes-Kasparov conjecture, Question \ref{quest} can be seen as the deformation of the $K$-theory of the group algebra $C^*_r(G_1)$, see \cite{Higson2008}, or the recent work \cite{Bernstein:2017aa}.

\section{Geometric quantization}

\subsection{General theory}

We first recall a few basic elements of the geometric quantization needed here.
The procedure of quantization consists of assigning to a smooth function on a symplectic manifold $(M,\omega)$, an operator on a Hilbert space $\mathcal H$ with the following three properties.
For each $f: M\to \mathbb R$, there should be an operator $\hat f$ acting on $\mathcal H$ such that 
\begin{enumerate}
    \item the map $f\mapsto \hat f $ is linear;
    \item if $f$ is constant, $\hat f $ is a multiplication operator
    \item if $[f_1,f_2]= f_3$, then $[\hat f_1,\hat f_2]= -i\hbar \hat f_3$,
\end{enumerate} 
where $\hbar$ is the Planck constant.

The main ingredient in the program of geometric quantization for a symplectic manifold $(M,\omega)$ of dimension $2 n$ with symplectic (exact) $2$-form $\omega$ is to construct a Hermitian line bundle $\pi:B\to M $ with Hermitian connection $\nabla$ for which the curvature $2$-form is equal to $\hbar^{-1}\omega$. Locally, the connection $\nabla$, or covariant derivative $\nabla_{X}$ for a vector field $X$ on $M$, is given by 
\begin{align}
    \nabla_X(f)=  X(f)  - \frac{i}{\hbar} \theta(X)f\, ,  
\end{align}
where $f$ is a function on $M$ where the connection $1$-form $\theta$ can be shown to be be locally related \cite[Proposition 23.5]{Hall:2013aa} to $\omega$ as $d\theta= \omega$.
It is possible to construct such a line bundle on $(M,\omega)$ only is the symplectic manifold satisfies the integrality condition or is quantizable. 
Such a result is for example proved in \cite[Theorem 23.9, p.489]{Hall:2013aa} where the integrality condition  is 
\begin{align}
    \frac{1}{2\pi \hbar} \int\limits_{\Sigma}\omega \in\ZZ\, ,
    \label{integral-condition}
\end{align}
where $\Sigma$ arbitrary closed $2$-surface in $M$.
This condition is purely geometric and involves the computation of the holonomy of the line bundle for closed loop, which when contracted must be equal to $1$. 
The condition \eqref{integral-condition} is obtained after using Stokes's theorem. 

This line bundle has a natural Hilbert space $\mathcal{H}$ consisting of square-integrable global sections with respect to the natural Liouville measure of $M$ and a Hermitian structure on the line bundle $B$. It is an inner product in the fibres denoted by $(\cdot,\cdot)$ with the property to be smooth in the sense that for a smooth section $v:M\to B$, $(v,v)$ is a smooth function of $M$. 
We now have the Hilbert space and we need to find the map $f\mapsto \hat f$ satisfying the three conditions. 
It is given, for a section $s:M\to B$, by 
\begin{align}
    \widehat{f}(s)=-i\hbar\nabla_{X_f}s+fs\, ,
\end{align}
where $X_f$  is the Hamiltonian vector field associated with the function $f$, i.e. $\omega(X_f,Y)=df(Y)$ for any vector field $Y$, and $fs$ is the multiplication operation. 
This operator obviously satisfies $(1)$ and $(2)$, the third condition is satisfied by using the properties of the covariant derivative $\nabla_{X_f}$, see \cite{Woodhouse:1980aa,Hall:2013aa} for more details. 

We will now follow \cite{Woodhouse:1980aa} to expose a dynamical interpretation of $\hat f$, which will give us a useful formula for later. 
Let $V_f$ be the vector field on $B$ given in a local trivialisation by
\begin{align}
    V_f=X_f+\hbar^{-1} L\frac{\partial}{\partial\phi}\, ,
\end{align}
where $z=re^{i\phi}$ is the coordinate on the fibre of $L$ and
\begin{align}
    L= \theta(X_f)-f
    \label{EQ:Lagrangian}
\end{align}
is the Lagrangian associated to the function $f$, of Legendre transform of the Hamiltonian $f$.  
This vector field is constructed such that it is invariant under gauge transformations and $f\mapsto V_f$ is a Lie algebra isomorphism from $C^{\infty}(M)$ to a subalgebra of vector fields on $B$.
Let $\xi_t$ denote the flow of $V_f$ which projects to $\rho_t$, the flow of $X_f$. 
From these two flows, we can define the action of the flow $\rho_t$ on the sections $s: M\to B$ by mapping it to $\hat \rho_t$ such that 
\begin{align}
    \xi_t(\widehat{\rho}_ts(m))=s(\rho_tm)\, .
    \label{hatrho}
\end{align}
This formula shows that $\hat \rho_t$ is the correct 'pull-back'  of the sections such that the flow $\xi_t$ corresponds to only moving the base point with $\rho_t$. 
In a local trivialisation, the sections $s$ and $\widehat{\rho}_ts$ are represented by the complex functions $\psi$ and $\widehat{\rho}_t\psi$ with the following relation
\begin{align} \label{EQ:action-can-transf}
    \widehat{\rho}_t\psi(m)=\psi(\rho_tm)\exp\left(-\frac{i}{\hbar}\int\limits^{t}_0 L(\rho_{t'} m)d t'\right)\, .
\end{align}

It might happen that the quantum Hilbert space is further reducible.
A notion of polarization imposes further restrictions.
\begin{defn} A polarization $\mathcal{P}$ of a symplectic manifold $M$ is a choice at each point $z\in M$ of a Lagrangian subspace $P(z)\subset T^{\C}_z M$ (dimension is $n$) satisfying the following conditions
\begin{enumerate}
\item If two complex vector fields $X$ and $Y$ lie in $P(z)$ at each point $z$, then the commutator $[X,Y]\in P(z)$.
\item The dimension of $P(z)\cap \overline{P(z)}$ is constant.
\end{enumerate}
\end{defn}

\begin{defn}[The quantum Hilbert space]
Let $\mathcal{P}$ be a polarization of $M$. A smooth section $s$ of $L\to M$ is said to be polarized if
\begin{align}
    \nabla_{X}s=0,\quad \forall X\in \overline{P}\, .
\end{align}
The quantum Hilbert space $\mathcal{H}_{\mathcal{P}}$ associated with $\mathcal{P}$ is the closure in the prequantum Hilbert space of the space of smooth, square integrable, polarized sections $s$ of $L\to M$.
\end{defn}
\begin{defn} A polarization $\mathcal{P}$ of a symplectic manifold $M$ is called purely complex if 
$$
P(z)\cap \overline{P(z)}=\varnothing
$$
for every $z\in M$.
\end{defn}
A purely complex polarization $\mathcal{P}$ yields existence of a unique integrable almost complex structure $J$ on $M$ such that $P(z)$ is the $(1,0)$-tanget space of $M$ for all $z\in M$.
\begin{defn}Let $\mathcal{P}$ be a purely complex polarization of $M$ and let $J$ be the integrable almost complex structure associated with $\mathcal{P}$ such that
$J(z)=i\Id$ on $P(z)$ and $J(z)=-i\Id$ on $\overline{P(z)}$. We say that $\mathcal{P}$ is a K\"ahler polarization if the bilinear form
$$
g(X,Y):=\omega(X,J(z)Y),\quad X,Y\in T_zM,
$$
is positive definite for each $z\in M$.
\end{defn}
For a thorough exposition of  the geometric quantization we refer \cite{Hall:2013aa,Kostant:1970aa,Kirillov:1990aa,Woodhouse:1980aa}.

\subsection{Lie group representation}

We now specify this construction to Lie group representations. 
Let $G$ be a locally compact Lie group with Lie algebra $\mathfrak{g}$.
Each $\xi \in\mathfrak{g}$ generates a right-invariant vector field $R_\xi$. 
Let $\mu \in \mathfrak{g}^*$ and $\O_\mu$ be the coadjoint orbit through $\mu$. 
A coadjoint orbit is a symplectic manifold with the symplectic form 
\begin{align}
    \omega_\mu(\xi, \eta)=\mu \left([\xi,\eta]\right)\, ,\quad \forall \xi,\eta \in\mathfrak{g}\, ,
\end{align}
corresponding to the reduction of the canonical symplectic form on $T^*G$ using the momentum map $J:T^*G\to \mathfrak g^*$. 
This form is exact and thus $\omega_\mu= d\theta_\mu$, where 
\begin{align}
    \theta_\mu(\xi) = \mu(\xi)\, .
    \label{Rxi}
\end{align}
$(\O_\mu,d\theta_\mu)$ is a presymplectic manifold the symplectic manifold of right cosets $G/G_\mu $, where $G_\mu$ is the stabilizer of $\mu$, i.e.
\begin{align}
    G_\mu:=\{g\in G\colon \Ad_g^*\mu=\mu\}\, ,
\end{align}
and its Lie algebra
\begin{align}
    \mathfrak g_\mu=\{\xi \in\mathfrak{g}\colon \ad_\xi^* \mu=0\}\, . 
\end{align}

A coadjoint orbit $\O_\mu$ is integral if the symplectic manifold $(\O_\mu,\omega_\mu)$ is integral in the sense of \eqref{integral-condition}.
Let $\O_{\mu}$ be an integral coadjoint orbit, then there exists \cite[Theorem 23.9, p.489]{Hall:2013aa} a  holomorphic line bundle $B$ over $\O_{\mu}$.
From \eqref{Rxi} the Lagrangian $L_{\mu}$ associated with $\mu$ is given by $L_\mu(\xi) = \frac12 \mu(\xi)$. 

Let $\pi\in\Gh$ and let $(U_{\alpha},\varphi_{\alpha})$ be a local trivialization of $\O_{\mu}$. Then it follows from \eqref{EQ:action-can-transf} that the flow $\rho_t$ acts locally on the space of global holomorphic sections of $B\to \O_{\mu}$ via the formula
\begin{align} \label{EQ:G-action}
    \pi(g) s(m)=s(\Ad^*_{g^{-1}}m )e^{\frac{i}{\hbar}\mu\int_m^{m_1} L_{\mu}\,ds},\,m\in\O_{\mu} . 
\end{align}
where $m_1 = \Ad^*_{g^{-1}}m$.  
We used the fact that the Lagrangian $L_{\mu}$ is invariant under the coadjoint flow and the flow on the base is given by the coadjoint action, see \cite{Marsden:2010aa}[Chapter 13]. 

The second element of the construction is a choice of polarization. 
A remarkable result \cite{Vergne1970a,Vergne1970b} yields existence of polarizing subalgebras for solvable Lie algebras. 
These subalgebras can be constructed algorithmically \cite{Oussa:2015aa}. 
It is unknown to the authors whether there exist complex polarizations for general connected Lie groups. Therefore, we make the following assumption.
\begin{assumption}\label{assumption} Let $G$ be a locally compact Lie group such that for every coadjoint orbit $\O$ there exists invariant K\"ahler  polarization $\mathcal{P}\subset T\O$.
\end{assumption}

\section{Contractions}

\subsection{Lie group contractions}

Before using geometric quantization to contract unitary duals, let's recall classic results on Lie group contractions. 

Let $G_1$ and $G_0$ be two locally compact connected Lie groups of the same dimension.  
\begin{defn}
\label{DEF:contraction}
We say that the family $\{\Phi_{\eps}\}_{\eps\geq0}$ of differential maps
\begin{align}
\Phi_{\eps}\colon G_0 \to G_1\, ,
\end{align}
mapping the identity $e_{G_0}$ of $G_0$ to the identity $e_{G_1}$ of $G_1$, defines a {\it contraction} of $G_1$ to $G_0$, if given any relatively compact open neighbourhood $V$ of $e_{G_0}$
\begin{enumerate}
    \item\label{1} there is $\eps_{V}>0$ such that for $\eps<\eps_{V}$ the map $\Phi_{\eps}\big|_{V}$ is a diffeomorphism;
    \item\label{2} If $W$ is such that $W^2\subset V$ and $\eps<\eps_{V}$ then $\Phi_{\eps}(W)^2\subset \Phi_{\eps}(V)$;
    \item\label{3} for $x,y\in V$
\begin{align}
    \lim_{\eps\to 0}\Phi^{-1}_{\eps}(\Phi_{\eps}(x) \Phi_{\eps}(y)^{-1})=x y^{-1}\, .
\end{align}
\end{enumerate}
\end{defn}
The first two conditions are necessary for the limit of the third condition to be well-defined. 

The Lie algebra can be contracted the differential $U_{\eps} = D_{e}\Phi_{\eps}$ of $\Phi_{\eps}$.
That is, at the identity $e_{G_0}$ 
\begin{align}
    U_{\eps}:= D_e\Phi_{\eps}: \g_0 \to \g_1\, .
\end{align}
This map is then used to contract the Lie algebra $U_\eps: \mathfrak g_0\to \mathfrak g_1$. 
Now, we fix a contraction $\Phi_{\eps}\colon G_0\to G_1$ of $G_1$ to $G_0$ as in Definition \ref{DEF:contraction}.

\subsection{Contraction of the unitary dual}
%}
Assume that there exists a contraction $\Phi_{\eps}\colon G_0\to G_1$. Then the differential $U_{\eps}$ of $\Phi_{\eps}$ allows us to deform the coadjoint orbits of $G_1$ into the coadjoint orbits of $G_0$. 
We express the adjoint action $\Ad_0$ of $G_0$ on $\mathfrak{g}_0$ as follows 
\begin{align}
    \Ad_0(e^X)=\lim_{\eps\to 0}U^{-1}_{\eps}\circ \Ad(e^{U_{\eps}(X)})\circ U_{\eps},\quad X\in \mathfrak{g}_0\, .
\end{align}
By the duality, we obtain the coadjoint action
\begin{align}
\label{EQ:Ad-0-lim}
\Ad^*_0(e^X)=\lim_{\eps\to 0}[U_{\eps}]^*\circ \Ad^*(e^{U_{\eps}(X)})\circ [U^{-1}_{\eps}]^*,\quad X\in \mathfrak{g}_0\, ,
\end{align}
where 
\begin{align}
    U^*_{\eps}(X)(Y)=Y(U_{\eps}(X)),\quad X,Y\in\mathfrak{g}^*_0\,.
\end{align}
The corresponding character of the contracted group $G_0$ is given by
\begin{align}
\chi_{\mu}(e^X)
=
\lim_{\eps\to 0}e^{U^*_{\eps}(f)(X)},\quad \mu\in\mathfrak{g}^*_0,\quad X\in\mathfrak{g}_0\, .
\end{align}

Following for example \cite{Arratia:1997ac}, the coadjoint orbits of the contracted group $G_0$ can be obtained via the following procedure
\begin{enumerate}
\item Fix an element $\mu\in \mathfrak{g}^*_0$ and fix a coadjoint orbit $\O_{\mu}$.
\item Transport the point $\mu$ to $\mathfrak{g}^*_1$ by using the map $(U^*_{\eps})^{-1}$.
\item The action of $G_1$ on $\mu_{\eps}=(U^*_{\eps})^{-1}(\mu)$ originates a $G_1$-coadjoint orbit $\tilde{\O_{\eps}}\subset \mathfrak{g}^*_1$.
\item The orbit $\tilde{\O_{\eps}}$ is translated to $\mathfrak{g}^*_0$ by using the map $U^*_{\eps}$, i.e.
\begin{align}
\O^{\eps}_{\mu}=U^*_{\eps}(\tilde{\O_{\eps}})\, .
\end{align}
\end{enumerate}
The orbit method of Kirillov is a method in geometric representation theory which identifies the unitary representations of Lie groups with the canonical $G$-action on spaces of sections of specific line bundles
over the coadjoint orbits.
The general case of a Lie group is not yet fully understood. Vogan rigorously justified \cite{Vogan:1998aa} the orbit method for reductive Lie groups.
However, we shall restrict ourselves to solvable type I Lie groups to use results of Kostant \cite{Kostant:1970aa} and Auslander \cite{Auslander:1967aa}.

Coadjoint orbits can be described by $\Ad_{G}$-invariant functions $f\colon \mathfrak{g}^*\to \C$. However,  it is not clear whether the holomorhic line bundles over the coadjoint orbits $\O_{\eps}$ are contracted properly. Therefore, we introduce
\begin{defn} 
\label{DEF:orbit-contraction}
Let $(U_{\alpha},\varphi_{\alpha})$ be a local trivialization of the line bundle $\O_{\mu}\to B$.
We shall say that a family $\{\O_{\eps}\}_{\eps>0}$  of integral polarized coadjoint orbits $\O_{\eps}$ is an admissible contraction of $\O_{\mu}$ if there is a family of local trivializations $(U_{\alpha},c^{\eps}_{\alpha})$ of the line bundles $\O_{\eps}\to B_{\eps}$
such that
\begin{align}
\label{EQ:contractibility}
\lim_{\eps\to 0}c^{\,\eps}_{\alpha\,\beta}=c_{\alpha\,\beta},\quad z\in V\, ,
\end{align}
uniformly on compact subsets $V\subset \mathbb{C}^r$,
where $c^{\,\eps}_{\alpha\,\beta}$ and $c_{\alpha\,\beta}$ are the transition functions associated with the trivializations $(U_{\alpha},\varphi_{\alpha})$ and $(U_{\alpha},\varphi^{\,\eps}_{\alpha})$.
\end{defn}
\begin{defn}  \label{DEF:convergence_sections}
Let $B\to M$ be a holomorphic line bundle. We say a sequence of global sections $\{s_n\colon M\to B\}$ converges to global section $s\colon M\to B$ if
\begin{align}
    \lim_{n\to\infty}c^n_{\alpha\,\beta}(x) s_{n\,\beta}(x)=s_{\alpha}(x),\quad x\in U_{\alpha\,\beta}\, ,
\end{align}
where $s_{n\,\beta}$ and $s_{\alpha}$ are the localization of $s_n$ and $s$ on $U_{\beta}$ and $U_{\alpha}$ respectively.
\end{defn}

A rigorous notion of contraction of representations was first introduced in \cite{Mickelsson:1972aa}.
\begin{defn}[{{\cite{Mickelsson:1972aa}}}]
We say that a representation $\sigma$ of $G_0$ in $\H$ is $MN$-contraction of the sequence $\{\pi^n\}$ of representations $\pi^n$ of $G_1$ in $\H^n$ if there exists a sequence $\eps_n$ with limit $0$, a sequence of unitary operators $A_n\colon \H^{n}\to \H$ and a dense subspace $D$ of $\H$ such that
\begin{enumerate}
\item for each $v\in D$ there is an integer $n(v)$ such that $n\geq n(v)$ implies $v\in A_n(\H^{n})$;
\item for each $v\in D$ and $g\in G_0$
\begin{equation}
\lim_{n\to 0}\|A_n\pi^n(\Phi_{\eps_n}(g))A^{-1}_nv-\sigma(g)v\|_{\H}=0.
\end{equation}
\end{enumerate}
\end{defn}

\begin{question}[{{\cite{Cahen2010}}}] Let $G_1$ be a semisimple compact Lie group. Assume that there is a group contraction of $G_1$ to $G_0$. What unitary representations $\pi$ of $G_0$ are $MN$-contractions of the unitary irreducible representations of $G_1$?
\end{question}
In Theorem \ref{THM:approx_space} we find  sufficient condition on elements $\pi\in\Gh$ to be contraction of a sequence $\{\pi_{\eps}\}\subset\widehat{G_1}$.
The condition is expressed in terms of the corresponding coadjoint orbits.
\begin{thm} 
\label{THM:approx_space}
Let $G_1$ be a compact Lie group and $G_0$ be a locally compact solvable type I Lie group.  Let $\pi\in\G_0$ be an infinite-dimensional irreducible unitary representation of $G_0$.
Then $\pi$ is an $MN$-contraction of the sequence of unitary irreducible representations if $\O_{\eps}$ is an admissible contraction of $\{\O_{\pi}\}$.
\end{thm}
\begin{question}
\label{QUEST:converse}
 Is the converse true?
In other words, let $\pi$ be an $MN$-contraction of the sequence of unitary representations $\pi_{\eps}\in \widehat{G_1}$.
Then the coadjoint orbits $\O_{\eps}$ is an admissible contraction of $\{\O_{\pi}\}$.
\end{question}
\begin{proof}[Sketch of the argument for Question \ref{QUEST:converse}]
Thinking locally, we should have $MN$-contraction on each $U_{\alpha}$, i.e.
$$
\pi_{\eps}(g)s_{\alpha} 
$$
on each local section $s_{\alpha}$. Each local section give rise to cocycles $c_{\alpha,\beta}$. 
Adding a bit more mathematical rigour one should be able to conclude the statement of Question \ref{QUEST:converse}.
\end{proof}

Prior to the proof of Theorem \ref{THM:approx_space}, we consider the classical contraction $\Phi_{\eps}\colon \mathbb{H}^1 \to \mathrm{SU}_2$. 
\begin{ex} 
\label{EX:SU2-H-contraction}
Let $G_1=SU(2)$ and $G_0=\HH$. A Lie group contraction $\Phi_{\eps}=\exp(U_{\eps})\colon G_0\to G_1$ is given by
\begin{align}
U_{\eps}=
\begin{psmallmatrix}
\eps^{\frac12} & 0 & 0 \\
0 & \eps & 0 \\
0 & 0 & \eps \\
\end{psmallmatrix}.
\end{align}
The coadjoint orbits $\{\O_s\}$ of $\SU2$ are the $2$-dimensional spheres with half-integer radius $s\in\frac12\NN$, i.e.
\begin{align}
\label{EQ:co-adjoint-orbits-SU2}
\O_s(\SU2)
=
\{
%F_{x_1,x_2,x_3}
%\in \mathfrak{su}(2)^*
(x_1,x_2,x_3)
\colon
x^2_1+x^2_2+x^2_3=s^2,\quad s\neq 0
\}
.
\end{align}
Using that $\O_s$ is diffeomorphic to the complex projective line 
$\mathbb{CP}^1$
we transport 
the coordinate atlas $U_1,U_2\subset \C$ on $\O_s$. 
There exists a holomorphic line bundle $\pi\colon B\to \O_s$ with the transition function  
\begin{align}
\label{EQ:SU-2-bundle-class}
c_s(z)=z^{2s}
\end{align}
defines a bundle $\pi\colon B\to \O_s$. The existence follows from the fact that $\O_s$ is a quantizable symplectic manifold. The choice \eqref{EQ:SU-2-bundle-class} picks up a specific bundle.

A global holomorphic section $s$ of the bundle $\pi\colon B\to \O_s$ is given by two holomorphic functions $s_1\colon U_1\to \C$ and $s_2\colon U_2\to \C$ related by
\begin{align}
    s_1\left (z\right )=c_s\left (z\right )s_2\left (\frac1z\right ),\quad z\in U_1\cap U_2\, .
\end{align}
Expanding holomorphic functions $s_1$ and $s_2$ into Taylor series
\begin{align}
\sum\limits_{n=0}a_mz^m
=
z^{2s}
\sum\limits_{n=0}\tilde{a}_nz^{-n}a
\end{align}
 in respective coordinates and equating coefficients, we find that $
 \tilde{a}_n=a_n=0$ for $m>2s$ and $\tilde{a}_0={a}_{2s},\tilde{a}_1={a}_{2s-1}$, etc\dots Hence, the section $s$ is given by a polynomial
\begin{align}
 s(z)= \sum\limits^{2s}_{n=0}a_nz^n\, .
\end{align}
The coadjoint orbits $\O_{\lambda}$ of the Heisenberg group $\HH$ are given by 
\begin{align}
\O_{\lambda}=\left \{z\in\C\colon \frac{z+z^*}2=\lambda,\quad \frac{z-z^*}2=0\right \}\, .
\end{align}
In other words, the orbits $\O_{\lambda}$ are planes through the point $\O_{\lambda}=\{(\lambda,0)\}$.
There exists a holomorphic line bundle $\pi\colon B\to \O_{\lambda}$ with the transition function 
\begin{align}
c^{\lambda}(z)=e^{i\lambda z},\,z\in\C\, .
\end{align}
A global holomorphic section $F$ of $\pi\colon B\to \O_{\lambda}$ is a holomorphic function.
We deform the transition function $c_s$ defined in \eqref{EQ:SU-2-bundle-class}
as follows
\begin{align}
c^{\eps}_s(z)
=
(1+i\eps z)^{2s_{\eps}}\, ,
\end{align}
where 
\begin{align}
s_{\eps}= \left[\frac{\lambda}{\eps}\right],\quad \lambda\in\RR\, .
\end{align}
An elementary calculation then yields
\begin{align}
\lim_{\eps\to 0}
(1+i\eps z)^{2s_{\eps}}
=
e^{i\lambda z}\, .
\end{align}
This shows the condition \eqref{EQ:contractibility} holds true.
\end{ex}
%\begin{ex}
%Let $G$ be a Lie group and $K$ is a closed subgroup of $G$. 
%Let $\mathfrak{g}$ and $\mathfrak{k}$ denote the Lie agebras of $G$ and $K$ respectively. 
%Using the linear action of $K$ on the vector space $V=\mathfrak{g}/\mathfrak{k}$, it is possible associate with $G$ a new Lie group called the semidirect product $K\times V$.
%This group is contraction \cite{DooleyRice1985} of $G$, i.e. there is a family  of differential maps
%\[
%\Phi_{\eps}\colon V\times K \to G,\quad vk\mapsto \exp_{G}(\eps v)\cdot k
%\]
%Mackey's normal subgroup analysis yields the description of the unitary dual of the semidirect product $K\times V$.
%\end{ex}
The extension of Example \ref{EX:SU2-H-contraction} to a broader class of Lie groups (e.g. semisimple Lie groups) remains a challenging problem \cite{Valette:2017aa} and will be considered elsewhere.
\begin{proof}[Proof of Theorem \ref{THM:approx_space}]
For every $g\in G_0$ the sections $\pi^{\eps}(g^{\eps})s^{\eps}$ converge to the section $\pi(g)s_{\alpha}$ in the sense of Definition \ref{DEF:convergence_sections}, i.e.
\begin{align}
\label{EQ:suff-approx-space}
\lim_{\eps\to 0}\pi^{\eps}(g^{\eps})s^{\eps}_{\alpha}=\pi(g)s_{\alpha}\, , 
\end{align}
uniformly for every compact subset $V\subset U_{\alpha}\cap U_{\beta}$.

Let $\pi$ be a fixed unitary irreducible representation of $\pi$ associated with a coadjoint orbit $\O_{\pi}\subset \g_0$. 
Since $\O^{\eps}_{\pi}$ is an admissible contraction of $\O_{\pi}$,
there exists an admissible contraction $\{\O^{\eps}_{\pi}\}_{\eps>0}$ and let us denote by $\pi_{\eps}$ be the representation corresponding to $\O^{\eps}_{\pi}$. 
Let $(U_{\alpha},\varphi_{\alpha})$ be a local trivialization of $\O_{\pi}\to B$ and 
let us fix an aribtrary element $s\in \H^{\pi}$. We denote by $s_{\alpha}$ its localization to $U_{\alpha}\subset O_{\pi}$. Abusing our notation we shall also write $s_{\alpha}$ for the coordinate representation of $s_{\alpha}$, i.e.
$
s_{\alpha}(z)=s_{\alpha}(\varphi^{-1}_{\alpha}(z)),\quad z\in \C^r\, .
$
Let us denote
\begin{align}
\Ad^{*\,\eps}_g\colon \O_{\pi}\to \O^{\eps}_{\pi}\, , 
\end{align}
given by
\begin{align}
\Ad^{*\,\eps}_g(X)=U^*_{\eps}\Ad^{*\,1}_g (U^{-1}_{\eps})^*(X),\quad X\in\g_0\, .
\end{align}
Indeed, by \eqref{EQ:G-action}
\begin{align}
\pi^{\eps}(g) s_{\alpha}(m)
=
s_{\alpha}(\Ad^{*\,\eps}_{g}m)\psi_{\mu}(m,g)\,
%\exp
%\left(
%\frac{i}{2\hbar}
%\int
%\limits_{\gamma\colon m \to \Ad^{*\,\eps}_{g}m}\mathcal{L}^{\eps}_{\mu}\,ds
%\right)\, .
\end{align}
where we denote
\begin{equation}
\label{EQ:cocyle}
\psi_{\mu}(m,g)
=
\exp \left(
\frac{i}{2\hbar} \int \limits_{\gamma\colon m\to \Ad^{*}_{g}m}\mathcal{L}_{\mu}\,ds \right).
\end{equation}
Thus, we get
\begin{align}
\begin{split}
\left| \pi(g) s_{\alpha}(m) - \pi^{\eps}(\Phi_{\eps}(g)) s^{\eps}_{\alpha}(m) \right|
%&= \left| s_{\alpha}(\Ad^{*}_{g}m) \exp \left(
%\frac{i}{2\hbar} \int \limits_{\gamma\colon m\to \Ad^{*}_{g}m}\mathcal{L}_{\mu}\,ds \right)\right. \\
%& \hspace{40mm}\left . - s^{\eps}_{\alpha}(\Ad^{*\,\eps}_{\Phi_{\eps}(g)}m) \exp \left( \frac{i}{2\hbar} \int \limits_{\gamma_{\eps}\colon m\to \Ad^{*\,\eps}_{\Phi_{\eps}(g)}m}\mathcal{L}^{\eps}_{\mu}\,ds \right) \right| \\
%&=
%\left| \left( s_{\alpha}(\Ad^{*}_{g}m) - s^{\eps}_{\alpha}(\Ad^{*\,\eps}_{g}m) \right) \exp \left( \frac{i}{2\hbar} \int \limits_{\gamma\colon m\to \Ad^{*}_{g}m}\mathcal{L}_{\mu}\,ds \right)\right| \\
%&+ \left| s_{\alpha}(\Ad^{*\,\eps}_{\Phi_{\eps}(g)}m) \left( \exp \left( \frac{i}{2\hbar} \int \limits_{\gamma\colon m\to \Ad^{*}_{g}m}\mathcal{L}_{\mu}\,ds \right) \right .\right . \\
%&\hspace{40mm}\left . \left . - \exp \left( \frac{i}{2\hbar} \int \limits_{\gamma_{\eps}\colon m\to \Ad^{*\,\eps}_{\Phi_{\eps}(g)}m}\mathcal{L}^{\eps}_{\mu}\,ds \right) \right) \right| \\
%&
\leq \left| s_{\alpha}(\Ad^{*}_{g}m)- s^{\eps}_{\alpha}(\Ad^{*\,\eps}_{g}m)\right| 
+ C \left| 
\psi_{\mu}(m,g)
-
\psi_{\mu_{\eps}}(m,\Phi_{\eps_n}(g))
%\left( \exp \left( \frac{i}{2\hbar} \int \limits_{\gamma\colon m\to \Ad^{*}_{g}m}\mathcal{L}_{\mu}\,ds \right) 
%- 
%\exp \left( \frac{i}{2\hbar} \int \limits_{\gamma_{\eps}\colon m\to \Ad^{*\,\eps}_{\Phi_{\eps}(g)}m}\mathcal{L}^{\eps}_{\mu}\,ds \right) \right) 
\right|\,,
\end{split}
\end{align}
where we used \eqref{EQ:cocyle} and the fact that $\left|\psi_{\mu}(m,g)\right|\leq 1$.
This shows that
\begin{align}
\label{EQ:step-1}
\lim_{\eps\to 0}
\left|
s_{\alpha}(\Ad^{*}_{g}m)
-
s^{\eps}_{\alpha}(\Ad^{*\,\eps}_{g}m)
\right| 
\quad \text{uniformly in $U_{\alpha}$}\, .
\end{align}
The sections $s$ and $s^{\eps}$ are holomorphic. Hence, their localizations $s_{\alpha}$ and $s^{\eps}_{\alpha}$ admit Taylor expansions
\begin{align}
\label{EQ:s-alpha-expansion}
s_{\alpha}(z)
&=
\sum\limits_{n\in\NN^r}c^{\alpha}_k z^k, \,z\in\varphi(U_{\alpha})\subset \C^r\, ,
\\
\label{EQ:s-alpha-eps-expansion}
s^{\eps}{\alpha}(z)
&=
\sum\limits_{n\in\NN^r}c^{\alpha\,\eps}_k z^k, \,z\in\varphi(U_{\alpha})\subset \C^r\, .
\end{align}
Since every coadjoint orbit $\O^{\pi}{\eps}$ is compact, the Hilbert space $\H^{\eps}$ of global holomorphic sections $s^{\eps}$ is finite-dimensional in the view of the Riemann-Roch theorem. 
The dimension of the representation space is given \cite{Kirillov:2004aa} by
\begin{align}
\dim(\pi^{\eps})=\vol(\O^{\eps}_{\pi})\, .
\end{align}
Hence, we get
\begin{align}
\label{EQ:RRoch}
s^{\eps}_{\alpha}(z)
=
\sum
\limits_{\left|k\right|\leq N_{\eps}} c^{\alpha\,\eps}_{k}z^k\,  ,
\quad z\in\C^r\, ,
\end{align}
where $N_{\eps}\geq \dim(\pi^{\eps}))$.
Let us choose 
\begin{align}
c^{\alpha\,\eps}_k=c^{\alpha}_k,\quad \left|k\right|\leq N_{\eps}\, .
\end{align}
We have
\begin{align}
\begin{split}
\label{EQ:divide}
\left|
s_{\alpha}(\Ad^{*}_{g}m)
-
s^{\eps}_{\alpha}(\Ad^{*\,\eps}_{g}m) 
\right|
&=
\left|
s_{\alpha}(\Ad^{*}_{g}m)
-
s^{\eps}_{\alpha}(\Ad^{*}_{g}m)
+
s^{\eps}_{\alpha}(\Ad^{*}_{g}m)
-
s^{\eps}_{\alpha}(\Ad^{*\,\eps}_{g}m) 
\right|
\\
&\leq
\left|
s_{\alpha}(\Ad^{*}_{g}m)
-
s^{\eps}_{\alpha}(\Ad^{*}_{g}m)
\right|
+
\left|
s^{\eps}_{\alpha}(\Ad^{*}_{g}m)
-
s^{\eps}_{\alpha}(\Ad^{*\,\eps}_{g}m)
\right|\, .
\end{split}
\end{align}
Since $s^{\eps}_{\alpha}$ is a continuous function and using \eqref{EQ:Ad-0-lim}, we get that the second term in the last inequality of \eqref{EQ:divide} goes to zero as $\eps\to 0$.

Composing \eqref{EQ:s-alpha-expansion} and \eqref{EQ:s-alpha-eps-expansion}
and \eqref{EQ:RRoch}, we get
\begin{align}
\begin{split}
\left|
s_{\alpha}(\Ad^{*}_{g}m)
-
s^{\eps}_{\alpha}(\Ad^{*}_{g}m)
\right|
=
\sum\limits_{\substack{k\in\NN^r\colon \left|k\right|\geq N_{\eps}}}c^{\alpha}_k (\Ad^{*}_{g}m)^k\, .
\end{split}
\end{align}
The last sum converges to zero uniformly by the holomorphicity of $s_{\alpha}$. This establishes \eqref{EQ:step-1}.
By the existence of contraction, we immediately obtain that
\begin{align}
\lim_{\eps\to 0}
\int
\limits_{\gamma\colon m\to \Ad^{*}_{g}m}\mathcal{L}_{\mu}\,ds
=
\int
\limits_{\gamma_{\eps}\colon m\to \Ad^{*\,\eps}_{\Phi_{\eps}(g)}m}\mathcal{L}^{\eps}_{\mu}\,ds\, .
\end{align}
Thus, we have just shown \eqref{EQ:suff-approx-space}.

This completes the proof.
\end{proof}

\section{Two classes of contractions}

We exhibit here two type of contractions that can be implemented on any matrix Lie group and which encompass most of the contractions with a diagonal contraction matrix $U_\eps$. 
First note that in the limit $\eps\to 0 $, the important terms in the matrix $U_\eps$ will be the lowest powers in $\eps$. 
Notice also that $\eps$ is arbitrary, so only the relative scaling between terms is important. 
From these considerations, we will study the following two form of the contraction matrices $U_\eps$. 
\begin{align}
    U_{IW} &= \eps \mathrm{Id} + (1- \eps) u \\
    U_{anti-IW} &= \eps^2 \mathrm{Id} + \eps (1- \eps) u\,,
    \label{u-def}
\end{align}
where $u$ is a matrix to be determined and does not depend on $\eps$. 
The first contraction is the classical In\"on\"u-Wigner contraction and the second will be called the anti-In\"on\"u-Wigner contraction for reasons that will become obvious soon. 
First, we need to define a particular decomposition of the Lie algebra. 
\begin{defn}
    For a diagonal matrix $u$ of dimension $n= \mathrm{dim}(\mathfrak g_1)$, define the two subspaces of $V=\mathbb R^n$, the underlying vector space of $\mathfrak g_1$ by 
    \begin{align}
        uV_N= 0 \qquad \mathrm{and}\qquad uV = V_R\, . 
    \end{align}
    Furthermore, these subspaces satisfy 
    \begin{align}
        \g_1 = V_N\oplus V_R\, .
    \end{align}
\end{defn}
We now want to derive a geometrical condition on $u$ that will ensure that we obtain a contraction. 
For this, we will need the following object. 
\begin{defn}[IW-tensor]
    We the following $2$-tensor    
\begin{equation}
    T[u](\xi,\eta) := u^2[\xi,\eta] -u\left ( [u\xi , \eta] + [\xi,u\eta] \right) + [u\xi,u\eta]\, ,
    \label{IW-tensor}
\end{equation}
that we will call the IW-tensor. 
\end{defn}
In the case when $T[u](\xi,\eta)= 0 $, it reduces to the well-known Nijenhuis tensor of complex geometry.

We will now write the conditions on $u$ in term of this tensor and the subspaces defined above. 
\begin{thm}
    The following conditions on $u$ give valid Lie algebra contraction for a choice of splitting $\mathfrak g_1= V_N\oplus V_R$:
    \begin{enumerate}
        \item If $T[u](\xi,\eta)\in V_R$, $U_{IW}$ is an In\"on\"u-Wigner contraction,
        \item if $T[u](\xi,\eta)\in V_N$, $U_{anti-IW}$ is an anti-In\"on\"u-Wigner contraction.
    \end{enumerate}
\end{thm}
\begin{proof}
    We start by proving the IW contraction condition by rewriting 
    \begin{align*}
        U_{IW} = (1-\eps)(\lambda \mathrm{Id} + u)\, ,    
    \end{align*}
    where $\lambda= \frac{\eps}{1- \eps}$ converges to $0$ when $\eps\to 0 $.
    We first rewrite the Lie bracket using \eqref{u-def} to get
\begin{align}
    U^{-1} [U\xi,U\eta]&= U^{-1}\left \{ \lambda^2 [\xi,\eta] + \lambda \left ( [u\xi , \eta] + [\xi,u\eta] \right) + [u\xi,u\eta]\right \}\, .
\end{align}
    We then have to ensure that the inner term is in $V_R$, the space where $U$ is invertible in the limit $\eps\to 0 $.  
First notice that from the definition of the contraction \eqref{u-def} we have the relation $U^{-1} (\lambda \zeta) = \zeta- U^{-1} u\zeta$ for an arbitrary element $\zeta\in \mathfrak g$. 
We can thus rewrite the previous equation as
\begin{align*}
    U^{-1} [U\xi,U\eta]&= \lambda [\xi,\eta] +  [u\xi , \eta] + [\xi,u\eta] \\
    &+ U^{-1}\left \{ -\lambda u[\xi,\eta] -u\left ( [u\xi , \eta] + [\xi,u\eta] \right) + [u\xi,u\eta]\right \}\\
    &= \lambda [\xi,\eta] +  [u\xi , \eta] + [\xi,u\eta] - u[\xi,\eta]\\
    &+U^{-1}\left \{ u^2[\xi,\eta] -u\left ( [u\xi , \eta] + [\xi,u\eta] \right) + [u\xi,u\eta]\right \}\, .
\end{align*}
This computation supposes that $u$ is diagonal, and will not apply for more general contractions. 
The, as $U^{-1}$ is only defined on $V_R$ in the limit $\eps\to 0 $, the condition for $u$ to be a valid contraction, i.e. for the limit $\eps\to 0 $ to exist translates to ta condition on the IW-tensor,  
\begin{align}
    T[u](\xi,\eta) \in V_R\, .
\end{align}
The new Lie bracket is 
\begin{align}
    [\xi,\eta]_0= [\xi,\eta]' + u^{-1}T[u](\xi,\eta)\, , 
    \label{ad0}
\end{align}
where we defined
\begin{align}
    [\xi,\eta]' &:= [u\xi,\eta] + [\xi,u\eta] - u[\xi,\eta]\\
    T[u](\eta,\xi) &= [u\xi,u\eta] - u[\xi,\eta]'\, . 
\end{align}

We now turn to the condition for the anti-IW contraction where we can still use the equivalent formulation 
\begin{align}
    U_{anti-IW}= \eps(\lambda Id + u)\, , 
    \label{u-anti-lambda}
\end{align}
where $\lambda = \frac{\eps}{1-\eps}$. 
We then follow the same procedure as before. 
We first expand the bracket using \eqref{u-anti-lambda} to get 
\begin{align}
    U^{-1} [U\xi,U\eta]&= U^{-1}\left \{\eps^{2\alpha}\left ( \lambda^2 [\xi,\eta] + \lambda \left ( [u\xi , \eta] + [\xi,u\eta] \right) + [u\xi,u\eta]\right )\right \}\, .
\end{align}
We now have the relation $U^{-1}(\eps \lambda \zeta)= \zeta- U^{-1}(\eps u\zeta)$, so we can rewrite
\begin{align*}
    U^{-1} [U\xi,U\eta]&= \eps \left (\lambda [\xi,\eta] +  [u\xi , \eta] + [\xi,u\eta]\right ) \\
    &+ U^{-1}\left \{\eps^{2} \left ( -\lambda u[\xi,\eta] -u\left ( [u\xi , \eta] + [\xi,u\eta] \right) + [u\xi,u\eta]\right )\right \}\\
    &=\eps\left ( \lambda [\xi,\eta] +  [u\xi , \eta] + [\xi,u\eta] - \eps u[\xi,\eta]\right)\\
    &+U^{-1}\left \{ \eps^{2} \left (u^2[\xi,\eta] -u\left ( [u\xi , \eta] + [\xi,u\eta] \right) + [u\xi,u\eta]\right )\right \}\, .
\end{align*}
We can go further in the computation to extract non-vanishing terms by using the relation
$U^{-1}(\eps \lambda (\eps\lambda^{-1}\zeta))= \eps\lambda^{-1}\zeta- U^{-1}(\eps u\eps\lambda^{-1}\zeta)$ to get
\begin{align*}
    U^{-1} [U\xi,U\eta]&=\eps\left ( \lambda [\xi,\eta] +  [u\xi , \eta] + [\xi,u\eta] - \eps u[\xi,\eta]\right)\\
    &+\eps\lambda^{-1} \left (u^2[\xi,\eta] -u\left ( [u\xi , \eta] + [\xi,u\eta] \right) + [u\xi,u\eta]\right )\\
    &+U^{-1}\left \{ \eps^{2}\lambda^{-1} u \left (u^2[\xi,\eta] -u\left ( [u\xi , \eta] + [\xi,u\eta] \right) + [u\xi,u\eta]\right )\right \}\, .
\end{align*}
Not that $\lambda\propto \eps$ when $\eps\to 0 $, thus the limit $\eps\to 0 $ in the previous equation gives the new Lie bracket only in term of the IW-tensor
\begin{align}
    [\xi,\eta]_0= T[u](\xi,\eta)\, .
    \label{ad0-anti}
\end{align}
In this case, the condition for existence of the contraction is given by
\begin{align}
    T[u](\xi,\eta) \in V_N\, .
\end{align}
\end{proof}
We can make a couple of interesting remarks at this point.
\begin{enumerate}
    \item  If $T[u](\xi,\eta)= 0, \forall \xi,\eta\in \mathfrak g$, then the contracted Lie algebra is isomorphic to the original Lie algebra and $u$ is a Nijenhuis tensor. 
        This can be seen by the fact that $T[u](\xi,\eta)= 0 $ is equivalent to $[u\xi,u\eta]= u[\xi,\eta]'$, thus $u$ defines a homomorphism of Lie algebra between $[\cdot,\cdot]$ and $[\cdot,\cdot]_0= [\cdot,\cdot]'$. 

    \item If $[\xi,\eta]'= 0, \forall \xi,\eta\in \mathfrak g$, $u$ is a derivation of the Lie algebra $\mathfrak g$ as it will satisfy the Leibniz rule $u[\xi,\eta]= [u\xi,\eta] + [\xi,u\eta]$. 
        The new bracket is then simply $[\xi,\eta]_0= u^{-1}[u\xi,u\eta]$, but this would also mean that $u$ is invertible everywhere, thus the two Lie algebra would be isomorphic.
\end{enumerate}
From these two remarks we have that in order to have a contraction, $u$ must not be a homomorphism of Lie algebras nor a derivation of Lie algebras. 

We can even go further and give an extra condition on the splitting as well the type of Lie algebra that results from the contraction. 
\begin{thm}
    Given a finite dimensional complex Lie algebra $\mathfrak g_1$, for any subspace $\mathfrak u$ of a subalgebra of $\mathfrak g_1$, letting $V_R= \mathfrak u$ yields an IWS contraction to a semi-direct algebra and letting $V_N= \mathfrak u$ yields an anti-IWS contraction to a nilpotent Lie algebra. 
\end{thm}
\begin{proof}
    Let us first look at the anti-IWS contractions.
    Let $V_N\subset \mathfrak p$, so that we have $[V_N, V_N]\subset V_N$ and $V_N$ are nilpotent. 
    One can easily check that $T[u](\xi_N, \eta_N)= 0 $, and that 
    \begin{align*}
        T[u](\xi_R,\eta_R)= -[\xi_R,\eta_R]_R + [ \xi_R,\eta_R]= [ \xi_R,\eta_R]_N\subset V_N\, .
    \end{align*}
    The last commutation trivially gives $T[u](\xi_N,\eta_R)= [\xi_N,\eta_R]_R- [\xi_N,\eta_R]_R= 0 $. 
    We thus have that if $V_N\subset \mathfrak b$ then the contraction is anti-IWS. 

    We now look at the IWS contractions. We first have $T[u](\xi_R, \eta_R)= 0 $, then $T[u](\xi_R,\eta_R) = [\xi_R,\eta_R]_R \subset V_R$ and $T[u](\xi_N,\eta_R)= 0 $. 
    This gives the corresponding result. 
\end{proof}

\subsection*{Acknowledgments}

The authors would like to thank Anthony Dooley and Fulvio Ricci along with Michael Ruzhansky for their support and advice while this research was being done.
Alexis Arnaudon acknowledges partial support from an Imperial College London Roth Award.


\begin{thebibliography}{BHS17b}
\bibitem[Aus67a]{Auslander:1967aa}
L.~Auslander and B.~Kostant.
\newblock Quantization and representations of solvable {L}ie groups.
\newblock {\em Bull. Amer. Math. Soc.}, 73:692--695, 1967.

\bibitem[ADO97]{Arratia:1997ac}
O.~Arratia and M.~A. Del~Olmo.
\newblock Contraction of representations of {$1+1$} kinematical groups and
  quantization.
\newblock {\em Internat. J. Modern Phys. A}, 12(1):125--130, 1997.
\newblock IWCQIS 96 (Dubna, 1996).

\bibitem[AJ85]{A.H.Dooley:1985aa}
A.H.Dooley and J.W.Rice.
\newblock On contractions of semisimple lie groups.
\newblock {\em Trans. Amer. Math. Soc.}, 1985.

\bibitem[BHS16]{Bernstein:2016ab}
J.~Bernstein, N.~Higson, and E.~Subag.
\newblock {Algebraic Families of Harish-Chandra Pairs}.
\newblock {\em arXiv:1610.03435}, 2016.

\bibitem[BHS17a]{Bernstein:2017ab}
J.~Bernstein, N.~Higson, and E.~Subag.
\newblock {Contractions of Representations and Algebraic Families of
  Harish-Chandra Modules}.
\newblock {\em arXiv preprint arXiv:1703.04028}, 2017.

\bibitem[BHS17b]{Bernstein:2017aa}
J.~Bernstein, N.~Higson, and E.~Subag.
\newblock {Contractions of Representations and Algebraic Families of
  Harish-Chandra Modules}.
\newblock {\em arXiv preprint arXiv:1703.04028}, 2017.

\bibitem[Cah03]{Cahen:2003aa}
B.~Cahen.
\newblock {Contraction de {${\rm SU}(2)$} vers le groupe de {H}eisenberg et
  calcul de {B}erezin}.
\newblock {\em Beitr\"age Algebra Geom.}, 44(2):581--603, 2003.

\bibitem[Cah10]{Cahen2010}
B.~Cahen.
\newblock Some {R}emarks on the {N}otion of {C}ontraction of {L}ie {G}roup
  {R}epresentations.
\newblock {\em Mathematica Moravica}, 14(1):35--46, 2010.

\bibitem[Doo83]{Dooley:1983aa}
A.~H. Dooley.
\newblock Contractions of {L}ie groups and applications to analysis.
\newblock In {\em Topics in modern harmonic analysis, {V}ol. {I}, {II}
  ({T}urin/{M}ilan, 1982)}, pages 483--515. Ist. Naz. Alta Mat. Francesco
  Severi, Rome, 1983.

\bibitem[FP08]{Fialowski:2008ab}
A.~Fialowski and M.~Penkava.
\newblock {Formal deformations, contractions and moduli spaces of {L}ie
  algebras}.
\newblock {\em Internat. J. Theoret. Phys.}, 47(2):561--582, 2008.

\bibitem[Gro04]{Gromov:2004ab}
N.~Gromov.
\newblock {From Wigner-Ino\"n\"u group contraction to contractions of algebraic
  structures}.
\newblock {\em Acta Physica Hungarica A Heavy Ion Physics}, 19(3):209--212,
  2004.

\bibitem[Hal13]{Hall:2013aa}
B.~C. Hall.
\newblock {\em Quantum theory for mathematicians}, volume 267 of {\em Graduate
  Texts in Mathematics}.
\newblock Springer, New York, 2013.

\bibitem[Hig08]{Higson:2008ab}
N.~Higson.
\newblock The {M}ackey analogy and {$K$}-theory.
\newblock In {\em Group representations, ergodic theory, and mathematical
  physics: a tribute to {G}eorge {W}. {M}ackey}, volume 449 of {\em Contemp.
  Math.}, pages 149--172. Amer. Math. Soc., Providence, RI, 2008.

\bibitem[Hig11]{Higson:2011ab}
N.~Higson.
\newblock On the analogy between complex semisimple groups and their {C}artan
  motion groups.
\newblock In {\em Noncommutative geometry and global analysis}, volume 546 of
  {\em Contemp. Math.}, pages 137--170. Amer. Math. Soc., Providence, RI, 2011.

\bibitem[IW53]{Inonu:1953ab}
E.~Inonu and E.~P. Wigner.
\newblock On the contraction of groups and their representations.
\newblock {\em Proc. Nat. Acad. Sci. U. S. A.}, 39:510--524, 1953.

\bibitem[Kos70a]{Kostant:1970aa}
B.~Kostant.
\newblock Quantization and unitary representations.
\newblock In {\em Lectures in modern analysis and applications III}, pages
  87--208. Springer, 1970.


\bibitem[Kir90]{Kirillov:1990aa}
A.~A. Kirillov.
\newblock Geometric Quantization.
\newblock In {\em Dynamical Systems IV}, pages 137--172. Springer, 1990.

\bibitem[Kir04]{Kirillov:2004aa}
A.~A. Kirillov.
\newblock {\em Lectures on the orbit method}, volume~64.
\newblock American Mathematical Society Providence, 2004.

\bibitem[KN87]{Kirillov:1987aa}
A.~Kirillov and Y.~Neretin.
\newblock {The variety of $A_n$ of $n$-dimensional Lie Algebra Structures}.
\newblock {\em Amer. Math. Soc. Transl.}, 1987.


\bibitem[MN72]{Mickelsson:1972aa}
J.~Mickelsson and J.~Niederle.
\newblock Contractions of representations of de {S}itter groups.
\newblock {\em Comm. Math. Phys.}, 27(3):167--180, 1972.

\bibitem[MR10]{Marsden:2010aa}
J.~E. Marsden and T.~S. Ratiu.
\newblock {\em Introduction to Mechanics and Symmetry: A Basic Exposition of
  Classical Mechanical Systems}.
\newblock Springer Publishing Company, Incorporated, 2010.

\bibitem[Ous15]{Oussa:2015aa}
V.~Oussa.
\newblock Computing vergne polarizing subalgebras.
\newblock {\em Linear and Multilinear Algebra}, 63(3):578--585, 2015.

\bibitem[Ric86]{Ricci:1986aa}
F.~Ricci.
\newblock A contraction of {${\rm SU}(2)$} to the {H}eisenberg group.
\newblock {\em Monatsh. Math.}, 101(3):211--225, 1986.

\bibitem[RR86]{Ricci:1986ab}
F.~Ricci and R.~L. Rubin.
\newblock Transferring {F}ourier multipliers from {${\rm SU}(2)$} to the
  {H}eisenberg group.
\newblock {\em Amer. J. Math.}, 108(3):571--588, 1986.

\bibitem[Ver70a]{Vergne1970a}
M.~Vergne.
\newblock Construction de sous-alg\`ebres subordonn\'ees \`a un \'el\'ement du
  dual d'une alg\`ebre de {L}ie r\'esoluble.
\newblock {\em C. R. Acad. Sci. Paris S\'er. A-B}, 270:A173--A175, 1970.

\bibitem[Ver70b]{Vergne1970b}
M.~Vergne.
\newblock Construction de sous-alg\`ebres subordonn\'ees \`a un \'el\'ement du
  dual d'une alg\`ebre de {L}ie r\'esoluble.
\newblock {\em C. R. Acad. Sci. Paris S\'er. A-B}, 270:A704--A707, 1970.
\bibitem[Vog98a]{Vogan:1998aa}
D.~Vogan.
\newblock The method of coadjoint orbits for real reductive groups.
\newblock {\em Representation theory of Lie groups (Park City, UT, 1998)},
  8:179--238, 1998.
\bibitem[TV17]{Valette:2017aa}
T.~Schodler and A.~Valette.
\newblock private communication, 2017.



\bibitem[Woo80]{Woodhouse:1980aa}
N.~Woodhouse.
\newblock {\em Geometric quantization}.
\newblock The Clarendon Press, Oxford University Press, New York, 1980.
\newblock Oxford Mathematical Monographs.

\end{thebibliography}
\end{document}